\documentclass [11pt]{amsart}
\usepackage {amsmath, amssymb, amscd, mathrsfs, mathtools, hyperref, enumerate, url, color, comment}
\usepackage[text={6in,8.5in},centering,letterpaper,dvips]{geometry}
\usepackage[all,cmtip,knot]{xy}

\newtheorem {theorem}{Theorem}[section]

\newtheorem {proposition}[theorem]{Proposition}
\newtheorem {corollary}[theorem]{Corollary}

\newtheorem {definition}[theorem]{Definition}
\newtheorem {question}[theorem]{Question}

\theoremstyle{remark}
\newtheorem {remark}[theorem]{Remark}

\DeclareFontFamily{U}{mathx}{\hyphenchar\font45}
\DeclareFontShape{U}{mathx}{m}{n}{
      <5> <6> <7> <8> <9> <10>
      <10.95> <12> <14.4> <17.28> <20.74> <24.88>
      mathx10
      }{}
\DeclareSymbolFont{mathx}{U}{mathx}{m}{n}
\DeclareFontSubstitution{U}{mathx}{m}{n}
\DeclareMathAccent{\widecheck}{0}{mathx}{"71}
 
\def\zz {{\mathbb{Z}}}
\def\rr {{\mathbb{R}}}
\def\R {\rr}
\def\cc {{\mathbb{C}}}
\def\qq {{\mathbb{Q}}}

\def\L{\mathscr{L}}
\def\C{\mathcal{C}}

\def\L {\mathcal{L}}
\def\G {\mathcal{G}}
\def\Go {\G^{\circ}}

\def\Spin{\mathbb{S}}

\def\is{\mathscr{S}}

\DeclareMathOperator{\inte}{\operatorname{int}}
\DeclareMathOperator{\rk}{\operatorname{rk}}

\def\spc{\operatorname{Spin}^c}

\def\Inv{\operatorname{Inv}}

\def\SWF{\operatorname{SWF}}

\def\cmto{\widecheck{\mathit{CM}}}
\def\hmto{\widecheck{\mathit{HM}}}

\def\hmtilde{\widetilde{\mathit{HM}}}
\def\HFhat{\widehat{\mathit{HF}}}
\def\HFminus{\mathit{HF}^-}
\def\HFplus{\mathit{HF}^+}

\def\HFred{\mathit{HF}_{\operatorname{red}}}
\def\tH{\widetilde{H}}

\def\swf{\SWF}
\def\spinc{\mathfrak{s}}
\def\s{\spinc}

\def\tY{\widetilde{Y}}

\def\pml{p^\lambda}

\def\vml{W^\lambda}

\def\fpq{p/q}
\def\fpqprime{p'/q'}
\def\fqp{q/p}
\def\fqpprime{q'/p'}
\def\HFred{HF_{\textup{red}}}
\def\Z{\mathbb{Z}}

\begin{document}

\title{Floer homology and covering spaces}

\author[Tye Lidman]{Tye Lidman}
\thanks{The first author was partially supported by NSF grants DMS-1128155 and DMS-1148490.}
\address {Department of Mathematics, North Carolina State University, 2311 Stinson Drive \\  Raleigh, NC 27695}
\email {tlid@math.ncsu.edu}

\author[Ciprian Manolescu]{Ciprian Manolescu}
\thanks {The second author was partially supported by NSF grants DMS-1104406 and DMS-1402914.}
\address {Department of Mathematics, UCLA, 520 Portola Plaza\\ Los Angeles, CA 90095}
\email {cm@math.ucla.edu}

\begin{abstract}
We prove a Smith-type inequality for regular covering spaces in monopole Floer homology. Using the monopole Floer / Heegaard Floer correspondence, we deduce that if a $3$-manifold $Y$ admits a $p^n$-sheeted regular cover that is a $\zz/p\zz$-$L$-space (for $p$ prime), then $Y$ is a $\zz/p\zz$-$L$-space. Further, we obtain constraints on surgeries on a knot being regular covers over other surgeries on the same knot, and over surgeries on other knots.
\end {abstract}

\maketitle

\section {Introduction}
Monopole Floer homology \cite{KMbook} and Heegaard Floer homology \cite{HolDisk, HolDiskTwo} are two leading theories used to study three-dimensional manifolds. Recently, the two theories have been shown to be isomorphic, by work of Kutluhan-Lee-Taubes \cite{KLT1, KLT2, KLT3, KLT4, KLT5} and of Colin-Ghiggini-Honda \cite{CGH1, CGH2, CGH3} and Taubes \cite{Taubes12345}. Although Floer homologies have found many applications, their interaction with many classical topological constructions is still not fully understood. The purpose of the present paper is to study their behavior with respect to regular coverings. Coverings play a fundamental role in three-dimensional topology, particularly in view of the recent proof of the virtually fibered conjecture \cite{Wise, Agol}. Although our results are limited to covers between rational homology spheres, we expect that some of the techniques will extend to more general covers between three-manifolds.

Our model is the following well-known inequality, due to P. Smith \cite{SmithInequality, Floyd, BredonBook}. Suppose that a group $G$ of order $p^n$ (where $p$ is prime) acts on a compact topological space $X$, with $H_*(X; \zz/p\zz)$ finite dimensional. Let $X^G$ denote the fixed point set. The mod $p$ Betti numbers of $X$ and $X^G$ are then related by:
 \begin {equation}
 \label {eq:smith}
 \sum_i \dim  H_i(X^G; \zz/p\zz) \leq \sum_i \dim H_i(X; \zz/p\zz).
 \end {equation}
Seidel and I. Smith \cite{SeidelSmithLoc} proved that an analogue of \eqref{eq:smith} holds for Lagrangian Floer homology, under certain conditions. Specifically, they only considered the case $G = \zz/2\zz$, and assumed that the Lagrangians admit a stable normal trivialization. Hendricks \cite{Hendricks} used their result in the context of Heegaard Floer theory to show that the knot Floer homology of a knot $K \subset S^3$ has rank at most as large as the knot Floer homology of $K$ inside the double branched cover $\Sigma(K)$. 

Another natural setting in which one can hope to apply the Seidel-Smith inequality is the Heegaard Floer homology of covers. If $\widetilde Y \to Y$ is a covering of closed $3$-manifolds, one can obtain a Heegaard diagram for $\widetilde Y$ from a Heegaard diagram for $Y$; see the work of Lee and Lipshitz \cite{LeeLipshitz}. If $\widetilde Y \to Y$ is a double cover and we take suitable symmetric products of the Heegaard surfaces, we end up almost in the setting of Seidel-Smith. However, a stable normal trivialization does not exist for the Lagrangians under consideration, and this approach runs into difficulties. An alternative approach would be to use monopole Floer homology, as defined by Kronheimer and Mrowka \cite{KMbook}, instead of Heegaard Floer homology. Difficult issues related to equivariant transversality arise in this setting as well.

Our strategy is to work with another version of monopole Floer homology. In \cite{Spectrum}, the second author defined an invariant of rational homology spheres, which takes the form of an equivariant suspension spectrum. Specifically, given a rational homology sphere $Y$ equipped with a $\spc$ structure $\s$, one can associate to it an $S^1$-equivariant spectrum $\SWF(Y, \s)$. The construction uses finite dimensional approximation of the Seiberg-Witten equations, and skirts transversality issues. This makes it possible to prove the following:

\begin{theorem}\label{thm:Smith}
Suppose that $\widetilde{Y}$ is a rational homology sphere, $Y$ is orientable, and $\pi:\widetilde{Y} \to Y$ is a $p^n$-sheeted regular covering, for $p$ prime. Let $\spinc$ be a $\spc$ structure on $Y$. Then, the following inequality holds: 
\begin{equation}\label{eq:swfsmith}
\sum_i \dim \widetilde H_i(\SWF(Y,\spinc); \mathbb{Z}/p\mathbb{Z}) \leq \sum_i \dim \widetilde H_i(\SWF(\widetilde{Y},\pi^*\spinc); \mathbb{Z}/p\mathbb{Z}).
\end{equation}     
\end{theorem}

Observe that, under the assumptions of the theorem, $Y$ must be a rational homology sphere as well. Indeed, if $H^1(Y; \zz) \neq 0$ then there would exist a surjective homomorphism $\pi_1(Y) \to \zz$. Since $\pi_1(\widetilde{Y}) \subseteq \pi_1(Y)$ is a subgroup of finite index, the restriction of that homomorphism to $\pi_1(\widetilde{Y})$ would be nontrivial, which would contradict $b_1(\widetilde Y)=0$.

The proof of Theorem~\ref{thm:Smith} is not difficult. The Floer spectrum $\SWF$ is constructed from Conley indices in the finite dimensional approximations, and \eqref{eq:swfsmith} follows from an application of the classical Smith inequality \eqref{eq:smith} to these Conley indices.

Theorem~\ref{thm:Smith} becomes powerful in conjunction with our work from a previous paper \cite{Equivalence}. There, we proved that the monopole Floer homology of Kronheimer and Mrowka can be recovered from $\SWF(Y, \s)$. Specifically, if $Y$ is a rational homology sphere with a $\spc$ structure $\spinc$, we showed that there are isomorphisms
\begin{equation}
\label{eq:Equiv}
\hmto_*(Y,\spinc)  \cong \widetilde{H}^{S^1}_*(\SWF(Y,\spinc)), \ \ \ 
 \hmtilde_*(Y,\spinc) \cong \widetilde{H}_*(\SWF(Y,\spinc)).
 \end{equation}
Here, $\hmto$ is the ``to'' version of monopole Floer homology defined in \cite{KMbook}, $\hmtilde$ is the homology of the mapping cone of $U$ on the monopole Floer complex $\cmto$  (cf. \cite[Section 5.3]{Lee}, \cite[Section 8]{Bloom}), and $\tH^{S^1}_*$ denotes reduced equivariant (Borel) homology.

Thus, from Theorem~\ref{thm:Smith} we obtain the following inequality in monopole Floer homology: 
\begin{corollary}
\label{cor:tildeinequality}
Under the assumptions of Theorem~\ref{thm:Smith}, we have
\begin{equation}\label{eqn:hmsmith}
\dim \hmtilde(Y,\spinc;\mathbb{Z}/p\mathbb{Z}) \leq \dim \hmtilde(\widetilde{Y},\pi^*\spinc;\mathbb{Z}/p\mathbb{Z}).  
\end{equation}
\end{corollary}

Furthermore, by applying the work of Kutluhan-Lee-Taubes and of Colin-Ghiggini-Honda and Taubes on the monopole Floer / Heegaard Floer equivalence, we can rephrase Corollary~\ref{cor:tildeinequality} in terms of Heegaard Floer theory. Their results say that
\begin{equation}
\label{eq:Equiv2}
 \hmto(Y, \spinc) \cong \HFplus_*(Y, \spinc), \ \  \hmtilde(Y,\spinc) \cong \HFhat_*(Y, \spinc),
  \end{equation}
where $\HFplus$ and $\HFhat$ are two versions of Heegaard Floer homology from \cite{HolDisk}.  Thus, Corollary~\ref{cor:tildeinequality} turns into an inequality in Heegaard Floer homology.

\begin{corollary}
\label{cor:hatinequality}
Under the assumptions of Theorem~\ref{thm:Smith}, we have
\begin{equation}\label{eqn:hfsmith}
\dim \HFhat(Y,\spinc;\mathbb{Z}/p\mathbb{Z}) \leq \dim \HFhat(\widetilde{Y},\pi^*\spinc;\mathbb{Z}/p\mathbb{Z}).  
\end{equation}
\end{corollary}

One can adapt these arguments to obtain a similar result for $\HFred$, the reduced version of Heegaard Floer homology defined in \cite{HolDisk}:

\begin{theorem}
\label{thm:HFred}
Under the assumptions of Theorem~\ref{thm:Smith}, we have the inequality:
$$ \dim \HFred(Y, \s; \zz/p\zz) \leq  \dim \HFred(\widetilde{Y}, \pi^* \s; \zz/p\zz).$$
\end{theorem}

In \cite{GenusBounds}, Ozsv\'ath and Szab\'o defined a $\zz/p\zz$-$L$-space to be a rational homology sphere such that $\HFhat(Y; \zz/p\zz)$ has dimension $|H_1(Y;\mathbb{Z})|$ over $\zz/p\zz$ (for $p$ prime). Equivalently, $\HFhat(Y, \s; \zz/p\zz)$ should be one-dimensional for all $\spc$ structures $\s$. In terms of $\HFred$, this means that $\HFred(Y, \s; \zz/p\zz) = 0$ for all $\s$. Inequality~\eqref{eqn:hfsmith} (or, alternatively, Theorem~\ref{thm:HFred}) has the following immediate consequence:

\begin{corollary}
\label{cor:ZpL}
Suppose that $\pi:\widetilde{Y} \to Y$ is a regular $p^n$-sheeted covering of orientable $3$-manifolds, for $p$ prime. Then, if $\widetilde Y$ is a $\zz/p\zz$-$L$-space, so is $Y$.
\end{corollary}

The $\zz/p\zz$-$L$-spaces are of interest because of various geometric properties. For example, a $\zz/p\zz$-$L$-space $Y$ cannot support co-orientable taut foliations, and no contact structure on $Y$ admits a symplectic filling with $b_2^+ > 0$; see \cite[Proof of Theorem 1.4]{GenusBounds}. 

A more natural notion is that of an $L$-space \cite{OSLens}, which is defined to be a rational homology sphere $Y$ with the property that $\HFhat(Y,\spinc)$ is free Abelian of rank $|H_1(Y;\mathbb{Z})|$. An $L$-space is a $\zz/p\zz$-$L$-space for all $p$. Examples of $L$-spaces include all elliptic manifolds \cite[Proposition 2.3]{OSLens}, all double branched covers of $S^3$ over quasi-alternating links \cite{BrDCov}, and many others. In the context of monopole Floer homology, $L$-spaces were studied in \cite{KMOS}.

Corollary~\ref{cor:ZpL} has the following implication with regard to $L$-spaces:

\begin{corollary}\label{cor:coversL}
Suppose that $\pi:\widetilde{Y} \to Y$ is a regular covering of orientable $3$-manifolds, such that $\widetilde{Y}$ is an $L$-space, and the group of deck transformations is solvable. Further, suppose that for any intermediate cover $Y'$ (i.e., such that there exist possibly trivial covers $\widetilde{Y} \to Y'$ and $Y' \to Y$), the group $\HFhat(Y')$ is torsion-free. Then $Y$ is an $L$-space.   
\end{corollary}

The torsion-free assumption in Corollary~\ref{cor:coversL} is not unreasonable. In fact, there are no known examples of rational homology spheres with $\HFhat$ containing torsion.  (For the first examples of $\mathbb{Z}$-torsion in Heegaard Floer homology, see \cite{JabukaMark}. Those examples have $b_1 > 0$.) Potentially, the notions of $L$-space and $\zz/p\zz$-$L$-space are the same.

Ozsv\'ath and Szab\'o asked whether, among prime, closed, connected 3-manifolds,  $L$-spaces are exactly those that admit no co-orientable taut foliations. In a similar vein, Boyer, Gordon, and Watson \cite{BoyerGordonWatson} conjectured that an irreducible rational homology 3-sphere is an $L$-space if and only if its fundamental group is not left-orderable. It is worth noting that both of these potential alternate characterizations of $L$-spaces behave well under taking covers. Indeed, suppose that $\widetilde Y \to Y$ is a covering between compact $3$-manifolds. If $Y$ has a co-orientable taut foliation, then so does $\widetilde Y$. Also, if $\pi_1(Y)$ admits a left-ordering, then so does its subgroup $\pi_1(\widetilde Y)$. In view of these observations, the following question was raised in \cite{BoyerGordonWatson}:

\begin{question}[Boyer-Gordon-Watson, \cite{BoyerGordonWatson}]
\label{que1}
If $\pi:\widetilde{Y} \to Y$ is a covering map, $Y$ is orientable, and $\widetilde{Y}$ is an $L$-space, does $Y$ have to be an $L$-space?
\end {question}

Our Corollary~\ref{cor:coversL} can be viewed as a partial answer to Question~\ref{que1}. More evidence for an affirmative answer to Question~\ref{que1} comes from manifolds with Sol geometry or with Seifert geometry, or more generally graph manifolds. For Seifert fibrations, the equivalence between $L$-spaces, non-left-orderable fundamental groups, and the absence of co-orientable taut foliations has already been established \cite{LiscaStipsicz3, BoyerRolfsenWiest, Peters, BoyerGordonWatson}.  This implies that the answer to Question~\ref{que1} is ``yes,'' provided that $Y$ has Seifert geometry.  If one works with $\zz/2\zz$-coefficients, the same holds for Sol geometry \cite{BoyerRolfsenWiest, BoyerGordonWatson} and graph manifolds \cite{BoyerClay, HRRW}.   

Note that the covering of the Poincar\'e homology sphere by $S^3$ is an example of a regular cover with non-solvable automorphism group, for which the conclusion of Corollary~\ref{cor:coversL} still holds. By taking intermediate covers corresponding to non-normal subgroups of the binary icosahedral group, we can also get examples of irregular covers between $L$-spaces.

\begin{remark}
It is worth pointing out that a rational homology sphere that covers an $L$-space is not necessarily an $L$-space. For example, consider the double cover of $S^2$ over itself, branched at two points. By introducing two orbifold points of type $1/4$ at the branch points on the base, and an orbifold point of type $2/3$ somewhere else on the sphere, we obtain a genuine double cover between $2$-orbifolds: 
$$S^2(2/3, 2/3, 1/2, 1/2) \to S^2(2/3, 1/4, 1/4).$$
This pulls back to a double cover between Seifert fibered rational homology spheres 
$$ M(-2; 2/3, 2/3, 1/2, 1/2) \to M(-1; 2/3, 1/4, 1/4),$$
in the notation of \cite{LiscaStipsicz3}. By the criterion in \cite[Theorem 1.1]{LiscaStipsicz3}, $M(-1; 2/3, 1/4, 1/4)$ is an $L$-space, whereas $M(-2; 2/3, 2/3, 1/2, 1/2)$ is not.
\end{remark}

In view of the inequality \eqref{eqn:hfsmith}, it is natural to ask the following strengthened version of Question~\ref{que1}:

\begin{question}
\label{que2}
If $\pi:\widetilde{Y} \to Y$ is a covering map between closed orientable $3$-manifolds, and $\spinc$ is a $\spc$ structure on $Y$, do we necessarily have 
$$ \rk \HFhat(Y,\spinc) \leq \rk \HFhat(\widetilde Y, \pi^*\spinc) \ \ ?$$
Here, $\rk$ denotes the rank of an Abelian group.
\end{question}

Some partial results along these lines, for double covers with $b_1 > 0$, were obtained by Lipshitz and Treumann using methods from bordered Floer homology \cite{LipshitzTreumann}.

We now turn to some concrete topological applications of our covering inequalities. For any family of rational homology spheres where we can obtain a good understanding of their monopole or Heegaard Floer homologies, we can look for obstructions to covering. For example, we have:
\begin{corollary}\label{cor:alternating-dbc}
Let $K$ be a hyperbolic alternating knot and let $L$ be any quasi-alternating link. Then, the double branched cover $\Sigma(L)$ is not an $r^n$-sheeted regular cover of $S^3_{p/q}(K)$ for any prime $r$.  
\end{corollary}
\begin{proof}
By \cite{OSLens}, $S^3_{p/q}(K)$ is not a $\zz/r\zz$-$L$-space for any $r$.  If $L$ is a quasi-alternating link, then as discussed above, $\Sigma(L)$ is an $L$-space; cf. \cite{BrDCov}. The result now follows from Corollary~\ref{cor:ZpL}.  
\end{proof}

\begin{remark}
It is interesting to compare Corollary~\ref{cor:alternating-dbc} to the case where $K$ is a non-hyperbolic alternating knot, i.e., $K = T(2,2n+1)$.  Then there are infinitely many lens space surgeries (namely those of the form $p/q$ where $|(4n+2)q - p| = 1$ \cite{Moser}). Any such lens space regularly covers infinitely many other lens spaces; in fact, we can find a cover of this form with any finite cyclic deck transformation group. Notice that lens spaces are branched double covers of two-bridge links with non-zero determinant, and such two-bridge links are quasi-alternating by \cite{BrDCov}.
\end{remark}

Let us focus further on manifolds obtained by Dehn surgery on knots in $S^3$. The Heegaard Floer homology of surgeries on knots can be computed in terms of the knot Floer complex; see \cite{IntSurg, RatSurg}. This yields a rank inequality between the reduced Heegaard Floer homologies of different surgeries. Using Theorem~\ref{thm:HFred}, we obtain:

\begin{theorem}\label{thm:surgerysameK}
Let $K$ be a non-trivial knot in $S^3$ and let $p,q,p',q'$ be positive integers.  If $\frac{p}{q} \leq 1$ and $\left \lceil \fqp \right  \rceil < \left \lfloor \fqpprime \right \rfloor$, then $S^3_{\fpq}(K)$ cannot be an $r^n$-sheeted regular cover of $S^3_{\fpqprime}(K)$ for any prime $r$.  
\end{theorem}

While Theorem~\ref{thm:surgerysameK} can be considerably strengthened, we work with it in its current incarnation to keep both the statement and proof simple. 

\begin{remark}
The condition that $K$ be non-trivial is clearly necessary, as for all non-zero $q \in \Z$ and $\fpqprime \in \mathbb{Q},$ the surgery $S^3_{1/q}(U) = S^3$ is a $p$-fold regular cyclic cover of $S^3_{\fpqprime}(U) = L(p',q')$.    
\end{remark}

\begin{remark}
If $K$ is the right-handed trefoil, then there are infinitely many pairs of surgeries for which one regularly covers the other with number of sheets a prime-power. Indeed, $p/q$ surgery on $K$ gives the lens space $L(p, 4q)$ when $p = 6q \pm 1$; cf. \cite[Proposition 3.2]{Moser}. Let $r^n$ be a prime power of the form $6k+1$ for a positive integer $k$. Then $S^3_{(6q \pm 1)/q}(K) = L(6q \pm 1, 4q)$ is a regular $r^n$-cover of $S^3_{(6q' \pm 1)/q'}(K) = L(6q' \pm 1, 4q')$ for $q'=q+k(6q\pm 1)$. Note that these examples have surgery coefficients greater than $1$, unlike in the statement of Theorem~\ref{thm:surgerysameK}. Similar examples can be found for other torus knots.
\end{remark}

The results and examples above raise the following question:

\begin{question}
\label{q1}
For what knots $K \subset S^3$ do there exist pairs of surgery coefficients $\frac{p}{q} \neq \frac{p'}{q'}$ such that $S^3_{p/q}(K)$ is a cover of $S^3_{p'/q'}(K)$?
\end{question}

Surgeries of this form can be called {\em virtually cosmetic}. Thus, Question~\ref{q1} generalizes the problem of characterizing all {\em cosmetic surgeries}, i.e., those with $S^3_{p/q}(K) \cong S^3_{p'/q'}(K)$. This is related to the cosmetic surgery conjecture, which asks if a non-trivial knot can have orientation-preserving homeomorphic surgeries \cite[Conjecture 6.1]{GordonICM}; for recent progress using similar techniques to those used here, see for example \cite{RatSurg, NiWu}.

In a different direction, we can also obtain obstructions to covering between surgeries on different knots. A simple class of examples come from {\em $L$-space knots}, which are knots for which some positive surgery is an $L$-space.  We have the analogous notion of $\zz/r\zz$-$L$-space knots.   
\begin{theorem}\label{thm:surgerylspaceK}
Let $K$ and $K'$ be non-trivial $\zz/r\zz$-$L$-space knots and $p,p',q,q'$ positive integers satisfying 
$$
(2g(K) - 1) \lceil  q/p \rceil < (2g(K') - 1) \lfloor q'/p' \rfloor,
$$ 
then $S^3_{p/q}(K)$ is not an $r^n$-sheeted regular covering of $S^3_{p'/q'}(K')$ for any prime $r$.
\end{theorem}

This paper is organized as follows. In Section~\ref{sec:spectrum} we outline the construction of the Seiberg-Witten Floer spectrum from \cite{Spectrum}. In Section~\ref{sec:proofcovers} we discuss covering spaces and prove Theorems~\ref{thm:Smith} and \ref{thm:HFred}, as well as their corollaries. In Section~\ref{sec:applications} we recall the knot surgery formula from \cite{RatSurg}, and use it to deduce Theorems~\ref{thm:surgerysameK} and \ref{thm:surgerylspaceK}.

\medskip
\noindent \textbf{Acknowledgements.} We thank Steve Boyer, Cameron Gordon, Robert Lipshitz, John Luecke, Alan Reid, Dylan Thurston and Liam Watson for helpful conversations. We are particularly grateful to Jianfeng Lin for suggesting the argument in the proof of Theorem~\ref{thm:HFred}.

\section{The Seiberg-Witten Floer spectrum}\label{sec:spectrum}
We review here the construction of the Seiberg-Witten Floer spectrum $\SWF(Y, \s)$, following \cite{Spectrum}. We mostly use the notational conventions from \cite{Equivalence}.  

\subsection{The Seiberg-Witten equations in global Coulomb gauge}
\label{sec:coulombs}
We will be studying the Seiberg-Witten equations on a tuple $(Y,g,\spinc,\Spin)$, where $Y$ is a rational homology three-sphere, $g$ is a metric on $Y$, $\spinc$ is a $\spc$ structure on $Y$, and $\Spin$ is a spinor bundle for $\spinc$.   We choose a flat $\spc$ connection $A_0$ on $\Spin$ which gives an affine identification of $ \Omega^1(Y; i \R)$ with $\spc$ connections on $\Spin$.  

Consider the configuration space 
\[
\C(Y,\spinc) = \Omega^1(Y; i\R) \oplus \Gamma(\Spin).
\]
The gauge group $\G = \G(Y):=C^\infty(Y,S^1)$ acts on $\C(Y,\spinc)$ by $u \cdot (a,\phi) = (a - u^{-1}du,u \cdot \phi)$. Since $b_1(Y)=0$, each $u\in \G$ can be written as $e^{f}$ for some $f: Y \to i\R$. We define the {\em normalized gauge group} $\Go$ to consist of those $u=e^{f} \in \G$ for some $f$ with $\int_Y f = 0$. 

Let $\rho:TY \otimes \C \to \text{End}(\Spin)$ be the Clifford multiplication.  For $\phi \in \Gamma(\Spin)$ we will write $(\phi \phi^*)_0$ for the trace-free part of $\phi \phi^* \in \Gamma( \text{End}(\Spin))$, and let 
$$\tau(\phi, \phi) = \rho^{-1}(\phi \phi^*)_0 \in \Gamma(iTY) \cong \Gamma(iT^*Y)=\Omega^1(Y; i\R).$$  Further, for $a \in \Omega^1(Y; i\R)$, we will use $D_a : \Gamma(\Spin) \to \Gamma(\Spin)$ to denote the Dirac operator corresponding to the connection $A_0 + a$, and $D$ for the case of $a = 0$.  

Inside of $\C(Y,\spinc)$ we have a {\em global Coulomb slice} to the action of $\Go$: 
$$W = \ker d^* \oplus \Gamma(\Spin) \subset \C(Y, \spinc),$$
where $d^*$ is meant to act on imaginary $1$-forms. 

For any integer $k$, we let $W_k$ denote the $L^2_k$ Sobolev completion of $W$. For $k \geq 5$, we consider the {\em Seiberg-Witten map}:
$$l + c: W_{k} \to W_{k-1},$$
where
\begin{eqnarray}
l(a,\phi) &=& (*da, D \phi) \label{eq:lmap} \\
c(a,\phi) &=& (\pi \circ \tau(\phi,\phi),\rho(a)\phi +  \xi(\phi)\phi) \label{eq:cmap},
\end{eqnarray}
where $\pi$ denotes the $L^2$ orthogonal projection to $\ker d^*$, and $\xi(\phi): Y \to i\rr$ is characterized by $d\xi(\phi) = (1-\pi)\circ \tau(\phi, \phi)$ and $\int_Y \xi(\phi) =0.$ Note that $l$ is a linear Fredholm operator, and $c$ is compact.  The Seiberg-Witten map is the gradient in an appropriate metric of the Chern-Simons-Dirac functional, $\L$, defined by 
\[
\L(a,\phi) = \frac{1}{2} \Bigl(\int_Y \langle \phi, D_a \phi \rangle  - \int_Y a \wedge da \Bigr).  
\]

Let $I \subset \R$ be an interval. If a map $\gamma: I \to W_k$ satisfies
\[
\frac{d}{d t} \gamma(t) = - (l + c)(\gamma(t)), 
\]
we say that $\gamma$ is a {\em Seiberg-Witten trajectory} (in Coulomb gauge). Such a trajectory $\gamma=(a(t), \phi(t)): \rr \to W_k$ is said to be {\em of finite type} if $\L(\gamma(t))$ and $\|\phi(t)\|_{C^0}$ are bounded in $t$.

\subsection{Finite-dimensional approximation} \label{sec:fdax}
For $\lambda > 1$, let us denote by $\vml$ the finite-dimensional subspace of $W$ spanned by the eigenvectors of $l$ with eigenvalues in the interval $(-\lambda,\lambda)$. The $L^2$ orthogonal projection from $W$ to $\vml$ will be denoted $\tilde{p}^\lambda$. We modify this to make it smooth in $\lambda$, by defining:
\begin{equation}
\label{eq:plprel}
\pml = \int^1_0 \beta(\theta) \tilde{p}^{\lambda - \theta}_{-\lambda + \theta} d \theta, 
\end{equation}
where $\beta$ is a smooth, non-negative function that is non-zero exactly on $(0,1)$, and such that $\int_\rr \beta(\theta) d\theta =1$. Observe that the image of $\pml$ is the subspace $\vml$.

On $\vml$, we consider the flow equation
\begin{equation}
\label{eq:approxgrad}
\frac{d}{d t} \gamma(t)=-(l + \pml c)(\gamma(t)). 
\end{equation}

We refer to solutions of \eqref{eq:approxgrad} as  {\em approximate Seiberg-Witten trajectories}. 

Fix a natural number $k \geq 5$. There exists a constant $R > 0$, such that all Seiberg-Witten trajectories $\gamma: \rr \to W$ of finite type are contained in $B(R)$, the ball of radius $R$ in $W_k$.  The following is a corresponding  compactness result for approximate Seiberg-Witten trajectories:

\begin{proposition}[Proposition 3 in \cite{Spectrum}] \label{prop:proposition3}
For any $\lambda$ sufficiently large (compared to $R$), if $\gamma: \mathbb{R} \to \vml$ is a trajectory of the gradient flow $(l + \pml c)$, and $\gamma(t)$ is in $\overline{B(2R)}$ for all $t$, then in fact $\gamma(t)$ is contained in $B(R)$.
\end{proposition}

\subsection{The Conley index and the Seiberg-Witten Floer spectrum}
The Seiberg-Witten Floer spectrum will be defined by means of the Conley index, an important construction from dynamical systems \cite{ConleyBook}.  We briefly recall the relevant definitions.  

Let $\{ \phi_t\}$ be a one-parameter family of diffeomorphisms of a smooth manifold $X$.  For a subset $A \subseteq X$, define    
$$
\Inv(A,\phi) = \{x \in A \mid \phi_t(x) \in A \text{ for all } t \in \mathbb{R} \}.  
$$    
Note that if $A$ is compact, so is $\Inv(A, \phi)$.  We say that a compact set $\is \subseteq X$ is an {\em isolated invariant set} if there is a compact set $A$ such that $\is = \Inv(A,\phi) \subset \inte(A)$.  We call $A$ an {\em isolating neighborhood} for $\is$.  The Conley index of an isolated invariant set is roughly a pair consisting of an isolating neighborhood and the set of points where the flow exits the isolating neighborhood.  We make this precise.  

\begin{definition}
A pair $(N,L)$ with $N,L$ compact subsets of $X$ is called an {\em index pair} for an isolated invariant set $\is$ if \\
a) $\Inv(N-L, \phi) = \is \subset \inte(N - L)$, \\
b) for all $x \in N$, if $\phi_t(x) \not \in N$ for some $t > 0$, then there exists $0 \leq \tau < t$ with $\phi_\tau(x) \in L$, and we call $L$ an {\em exit set} for $N$, \\
c) for $x \in L$ and $t> 0$, if $\phi_s(x) \in N$ for all $0 \leq s \leq t$, then $\phi_s(x) \in L$ for $0 \leq s \leq t$ and we say $L$ is {\em positively invariant} in $N$.   
\end{definition}

Conley proved that any isolated invariant set admits an index pair \cite{ConleyBook}.  We define the {\em Conley index} of an isolated invariant set $\is$ to be the pointed space $(N/L, [L])$ for an index pair $(N,L)$.  This is denoted $I(\phi, \is)$, and its pointed homotopy is independent of the choice of index pair (although it depends heavily on the choice of $\is$).    If $\phi$ is $G$-equivariant for a compact Lie group $G$ acting on $X$, then a $G$-invariant index pair can be constructed \cite{FloerConley, Pruszko}, thus yielding a $G$-equivariant Conley index, denoted $I_{G}(\is)$.    

With this in mind, we are ready to define the Seiberg-Witten Floer spectrum.  We fix $k$, $R$, and sufficiently large $\lambda$ such that Proposition~\ref{prop:proposition3} applies.  We consider the vector field $u^\lambda(l + \pml c)$ on $\vml$, where $u^\lambda$ is a smooth, $S^1$-invariant, cut-off function on $W^\lambda$ that vanishes outside of $B(3R)$.  This generates the flow $\phi^\lambda$ that we will work with.  Denote by $S^\lambda$ the union of all trajectories of $\phi^\lambda$ inside $B(R)$. Recall from Proposition~\ref{prop:proposition3} that these are the same as the trajectories that stay in $\overline{B(2R)}$. This implies that $S^\lambda$ is an isolated invariant set. 

Since everything is $S^1$-invariant, we can construct the equivariant Conley index $I^\lambda = I_{S^1}(\phi^\lambda,S^\lambda)$.  We must de-suspend appropriately to make the stable homotopy type independent of $\lambda$:
\[
\SWF(Y,\spinc,g) =  \Sigma^{-W^{(-\lambda,0)}} I^\lambda,
\]
where $W^{(-\lambda, 0)}$ denotes the direct sum of the eigenspaces of $l$ with eigenvalues in the interval $(-\lambda, 0)$. As we vary the metric $g$, the spectrum $\SWF(Y, \spinc, g)$ varies by suspending (or de-suspending) with copies of the vector space $\cc$. In \cite{Spectrum}, this indeterminacy is fixed by introducing a quantity $n(Y,\spinc, g) \in \qq$ (a linear combination of eta invariants), and setting  
$$\SWF(Y, \spinc) = \Sigma^{-n(Y, \spinc, g) \cc} \SWF(Y, \spinc, g),$$
where the de-suspension by rational numbers is defined formally. We have:

\begin{theorem}[Theorem 1 in \cite{Spectrum}]
The $S^1$-equivariant stable homotopy type of $\SWF(Y,\spinc)$ is an invariant of the pair $(Y, \spinc)$.  
\end{theorem}

\section{Covering spaces}
\label{sec:proofcovers}

In this section we prove Theorems~\ref{thm:Smith} and~\ref{thm:HFred}, as well as Corollaries~\ref{cor:ZpL} and ~\ref{cor:coversL}. 

Suppose we are under the hypotheses of Theorem~\ref{thm:Smith}, with $Y$ and $\widetilde{Y}$ being rational homology spheres, and $\pi:\widetilde{Y} \to Y$ a $p^n$-sheeted regular covering, for $p$ prime. Let $\spinc$ be a $\spc$ structure on $Y$.

We equip $Y$ with a Riemannian metric $g$.  We use $\widetilde{g} = \pi^*g$ as our choice of metric on $\widetilde{Y}$.  Furthermore, we pull back the spinor bundle $\Spin$ from $Y$ to $\widetilde{Y}$.  We fix  a flat $\spc$ connection $A_0$ on $Y$ and consider its pullback $\widetilde{A}_0$ on $\widetilde{Y}$. In general, we will decorate an object with a tilde to mean the associated object for $\widetilde{Y}$.  Let $G$ denote the deck transformation group on $\widetilde{Y}$.  The main idea is that we can follow the constructions of the Seiberg-Witten Floer spectrum for $\widetilde{Y}$ such that it is a $G$-spectrum and at each step, the fixed points correspond exactly to the pull-backs of the corresponding objects on $Y$.  We can then use this to compare the Seiberg-Witten Floer spectra of $\widetilde{Y}$ and $Y$.    

Let $\widetilde{W}$ denote the Coulomb slice of the configuration space on $(\widetilde{Y},\pi^*\mathfrak{s})$.  The group $G$ acts on $\widetilde{W}$ by pull-back.  Furthermore, since the flat connection $A_0$ on $Y$ pulls back to the flat connection $\tilde{A}_0$ on $\widetilde{Y}$, we have that $\pi^*$ gives an inclusion of $W$ into $\widetilde{W}$ such that $W = \widetilde{W}^G$; being in Coulomb gauge is also preserved under pullbacks. We do point out that $\pi^*$ does not induce an isometric embedding from $W$ into $\widetilde{W}$.  This is because if $(a,\phi)$ has $L^2$ norm $1$ in $W$, then $(\pi^*(a),\pi^*(\phi))$ has $L^2$ norm $|G|$ in $\widetilde{W}$.  In particular, this shows that $B_W(R)$, the ball in $W$ of radius $R$, is precisely $B_{\widetilde{W}}(|G|\cdot R)^G$.  We then extend this to identify the Sobolev completion $W_{k}$ with $\widetilde{W}_{k}^G$.  Again, pullback dilates the $L^2_k$ norms by $|G|$.  

The linear Fredholm map $\widetilde{l}$ on $\widetilde{W}$ is $G$-equivariant. The map $\widetilde{l}$ can have more eigenvalues than $l$, but in any case any eigenvector of $l$ pulls back to an eigenvector of $\widetilde{l}$, for the same eigenvalue. Thus, for all $\lambda$, we have that $W^\lambda = (\widetilde{W}^\lambda)^G$.  We also have that $\widetilde{c}$ is $G$-equivariant. In particular, the pullback of a finite-type Seiberg-Witten trajectory on $Y$ in $W_{k}$ gives a $G$-invariant finite-type Seiberg-Witten trajectory on $\widetilde{Y}$ in $\widetilde{W}_{k}$. Similarly, the trajectories of $l + \pml c$ on $Y$ pull back to approximate Seiberg-Witten trajectories on $\widetilde{Y}$.   

Recall that for sufficiently large radius $R_Y$, all finite-type Seiberg-Witten trajectories on $Y$ are inside of $B_Y(R_Y)$.  By choosing each radius large enough, we can arrange that $R_Y = R_{\widetilde{Y}}/|G|$.   In particular, we have that 
$$(\widetilde{W}^\lambda \cap \overline{B_{\widetilde{Y}}(2R_{\widetilde{Y}}}))^G = W^\lambda \cap \overline{B_Y(2R_Y)}.$$  We choose $\lambda$ large enough such that Proposition~\ref{prop:proposition3} applies to each of $R_Y$ and $R_{\widetilde{Y}}$ on the respective manifold.  To obtain a well-defined flow on $W^\lambda$, we multiply $l + p^\lambda c$ by a bump function $\tilde{u}^\lambda$ on $B_{\widetilde{Y}}(3R_{\widetilde{Y}})$ which is radially symmetric in the $L^2_k$ norm.  Note that the $L^2_k$ norm on $\widetilde{W}_k$ is $G$-invariant by construction, and thus $\tilde{u}^\lambda$ is $G$-invariant.  We induce a corresponding bump function on $ B_{Y}(3R_{Y})$ by restriction. The truncated gradients induce the associated flows $\phi^\lambda$ and $\widetilde{\phi}^\lambda$ on $W^\lambda$ and $\widetilde{W}^\lambda$ respectively.  

Recall that the invariant set that we use for the Conley index on $\vml$ is $\is^\lambda$, the union of all finite-type trajectories of $\phi^\lambda$ in $W^\lambda \cap \overline{B(2R_Y)}$. On $\widetilde{W}^{\lambda}$ we have a similar isolated invariant set $\widetilde{\is}^\lambda$. Clearly, $(\widetilde{\is}^\lambda)^G = \is^\lambda$.  

The flow $\widetilde{\phi}^\lambda$ is not only $S^1$-equivariant, but also equivariant with respect to the action of $G$. Thus, we can choose an $(S^1 \times G)$-equivariant index pair $(N, L)$ for $\widetilde{\is}^\lambda$. The fixed point sets $(N^G,L^G)$ of the $G$-action form an $S^1$-invariant index pair for $\is^{\lambda}$. This implies that 
\begin{equation}
\label{eq:SmithConley}
I_{S^1 \times G}(\widetilde{\is}^\lambda, \widetilde{\phi}^\lambda)^G \simeq I_{S^1}(\is^\lambda,\phi^\lambda).
\end{equation}

To summarize, we can find based, compact $S^1$-spaces $$X=I_{S^1}(\is^\lambda,\phi^\lambda)$$ and $$\widetilde X=I_{S^1 \times G}(\widetilde{\is}^\lambda, \widetilde{\phi}^\lambda)$$ such that $\swf(Y, \s)$ and $\swf(\tY, \pi^* \s)$ are suitable (de-)suspensions of $X$ and $\widetilde X$, respectively; and moreover, $\widetilde X$ comes with an action of the group $G$ (commuting with the $S^1$-action), such that the $G$-fixed point set is $X$.

\begin{proof}[Proof of Theorem~\ref{thm:Smith}]
This follows by applying the classical Smith inequality \eqref{eq:smith} to the spaces $X$ and $\widetilde X$.  
\end{proof}

\begin{proof}[Proof of Theorem~\ref{thm:HFred}] All homologies below will be taken with coefficients in $\zz/p\zz$.

The variant $\HFred$ of Heegaard Floer homology \cite{HolDisk} can be described as the quotient $\HFminus/U^N \HFminus$ for $N \gg 0$. Using the Floer spectrum / monopole Floer / Heegaard Floer equivalences from \eqref{eq:Equiv} and \eqref{eq:Equiv2}, we see that $\HFred(Y, \s)$ is isomorphic (ignoring absolute gradings) to
$$ \tH_*^{S^1}(X) / \bigl( U^N \cdot \tH_*^{S^1}(X) \bigr), \ \ N  \gg 0.$$

Consider the long exact sequence in Borel homology associated to the pair $(X \wedge S(\cc^N)_+,$  $X \wedge D(\cc^N)_+)$. Observe that:
\begin{itemize}
\item The space $X \wedge S(\cc^N)_+$ has free $S^1$-action away from the basepoint, so its (reduced) Borel homology is isomorphic to the ordinary (reduced) homology of the quotient, $\tH_*(X \wedge_{S^1} S(\cc^N)_+)$;
\item The space $X \wedge D(\cc^N)_+$ is $S^1$-equivalent to $X$;
\item Smashing with $\bigl(D(\cc^N)_+/S(\cc^N)_+\bigr) \sim (\cc^N)^+$ preserves Borel homology (up to a degree shift by $2N$).
\end{itemize}
Thus, we can write the long exact sequence as:
$$ \dots \to \tH_*(X \wedge_{S^1} S(\cc^N)_+) \to  \tH_*^{S^1}(X) \to   \tH_{*-2N}^{S^1}(X) \to \dots$$
The map $\tH_*^{S^1}(X) \to   \tH_{*-2N}^{S^1}(X)$ in this sequence is induced from the composition
$$ X \hookrightarrow X \wedge D(\cc^N)_+ \to X \wedge \bigl(D(\cc^N)_+/S(\cc^N)_+\bigr) \to X \wedge (\cc^N)^+ \cong \Sigma^{n\cc}X.$$
Hence, on homology, the map is given by multiplication with the equivariant Euler class of $n\cc$, which is $U^N \in H^{2N}_{S^1}(\text{pt})$.

For $N$ large, multiplication by $U^N$ on $\tH_*^{S^1}(X)$ has kernel of dimension $N + \dim \HFred(Y, \s)$ and cokernel isomorphic to $\HFred(Y, \s)$. Therefore,
\begin{equation}
\label{eq:redX}
\dim \tH_*(X \wedge_{S^1} S(\cc^N)_+) = N + 2 \dim \HFred(Y, \s).
\end{equation}
Similar arguments apply to $\widetilde Y$ and $\widetilde X$, giving
\begin{equation}
\label{eq:redXtilde}
\dim \tH_*(\widetilde X \wedge_{S^1} S(\cc^N)_+) = N + 2 \dim \HFred(\widetilde Y, \s).
\end{equation}

Note that the $G$-fixed point set of $\widetilde X \wedge S(\cc^N)_+$ is $X \wedge S(\cc^N)_+$. Applying the classical Smith inequality to these spaces, together with \eqref{eq:redX} and \eqref{eq:redXtilde}, yields the desired inequality between the dimensions of $\HFred$ for $Y$ and $\widetilde Y$.  
\end{proof}

\begin{proof}[Proof of Corollary~\ref{cor:ZpL}] This follows from Theorem~\ref{thm:HFred}, using the characterization of $\zz/p\zz$-$L$-spaces in terms of the vanishing of $\HFred$ with $\zz/p\zz$ coefficients.
\end{proof}

\begin{proof}[Proof of Corollary~\ref{cor:coversL}] This is immediate from Corollary~\ref{cor:ZpL} and the universal coefficient theorem.
\end{proof}

\begin{remark}
Apart from Smith-type inequalities, the use of the Seiberg-Witten Floer spectrum allows us to define equivariant Seiberg-Witten Floer homologies for covering spaces. Indeed, given a regular cover $\pi : \widetilde Y \to Y$ between rational homology spheres, with any deck transformation group $G$, and equipped with a $G$-invariant $\spc$ structure $\s$, we set
$$  \mathit{SWFH}_*^{G}(\widetilde Y, \s) = \widetilde{H}_*^{G}(\swf(\widetilde Y, \s)), \ \ \ \ \mathit{SWFH}_*^{G \times S^1}(\widetilde Y, \s) = \widetilde{H}_*^{S^1 \times G}(\swf(\widetilde Y,\s)).$$
These invariants are modules over the rings $H^*(BG)$ and $H^*(B(G \times S^1)) = H^*(BG)[U]$, respectively.
\end{remark}

\section{Applications}
\label{sec:applications}

It is clear that when combined with the following, Corollary~\ref{cor:hatinequality} proves Theorem~\ref{thm:surgerysameK}.  

\begin{proposition}\label{prop:sameKinequality}
Let $K$ be a non-trivial knot in $S^3$ and let $p,q,p',q'$ be positive, relatively prime integers.  If $p/q \leq 1$ and $\left \lceil \fqp \right  \rceil < \left \lfloor \fqpprime \right \rfloor$, then for all primes $r$ and for all $\spinc \in \spc(S^3_{\fpq}(K))$ and $\spinc' \in \spc(S^3_{\fpqprime}(K))$, we have
\begin{equation}
\label{eq:sameK} \dim \HFhat(S^3_{ \fpq}(K),\spinc;\zz/r\zz) < \dim \HFhat(S^3_{ \fpqprime}(K),\spinc'; \zz/r\zz).
\end{equation}

\end{proposition}  

We will establish the desired inequality by applying the formula of Jabuka \cite{Jabuka} (based on that of Ozsv\'ath-Szab\'o in \cite{RatSurg}) for the Heegaard Floer homology of $p/q$-surgery on a knot $K$ in $S^3$.  We first recall his notation.  The formula will be expressed in terms of two objects: $H_*(\widehat{A}_s)$, defined in \cite{IntSurg}, which represent the Heegaard Floer homology of large surgeries on $K$ in certain $\spinc$ structures, and $\nu(K)$, a $\zz$-valued invariant defined in \cite{RatSurg}.  While Jabuka's results are stated with $\zz$-coefficients, the arguments also work for $\zz/r\zz$-coefficients for any prime $r$; we will omit the coefficients from the notation.  We do not need the definitions of either $\widehat{A}_s$ or $\nu$, just the following three facts: 
\begin{enumerate}[(i)]
\item either $\nu(K)$ or $\nu(-K)$ is non-negative (where $-K$ denotes the mirror of $K$), 
\item $\dim H_*(\widehat{A}_s) \geq 1$ for all $s$,
\item $\dim H_*(\widehat{A}_s) = 1$ for all $s$ implies that $\nu > 0$ for any non-trivial knot.
\end{enumerate}  

For $[i] \in \zz/p\zz$ and $s \in \zz$, let $\phi^{p/q}_{[i]}(s) = \# \{ n \in \zz \mid \lfloor \frac{i + p \cdot n}{q} \rfloor = s \}$.  Here, we also allow $p < 0$.  It is straightforward to verify that 
\begin{equation}\label{eq:phi-inequality}
\lfloor |q/p| \rfloor \leq \phi^{p/q}_{[i]}(s) \leq \lceil |q/p| \rceil.
\end{equation}
Further, let 
$$\mathcal{S}^{p/q}_{[i]} = \sum_{s \in \zz} \phi^{p/q}_{[i]}(s) \left( \dim H_*(\widehat{A}_s) - 1 \right).$$

\begin{theorem}[Jabuka \cite{Jabuka}]\label{thm:jabuka}  
Fix relatively prime integers $p,q$ with $q > 0$, and a knot $K$ in $S^3$. After possibly mirroring $K$, we can arrange that $\nu = \nu(K) \geq 0$ and if $\nu > 0$, 
$$
\dim \HFhat(S^3_{p/q}(K), [i]) = 
\begin{cases} 1 + \mathcal{S}^{p/q}_{[i]} &\text{ if } 0 < (2\nu - 1) q \leq p, \\ 
-1 + 2\sum_{|s| < \nu} \phi^{p/q}_{[i]}(s) + \mathcal{S}^{p/q}_{[i]} & \text{ if } 0 < p \leq (2\nu - 1)q,  \\
1 + 2\sum_{|s| < \nu} \phi^{p/q}_{[i]}(s) + \mathcal{S}^{p/q}_{[i]} & \text{ if } p < 0,
 \end{cases}
$$
while if $\nu = 0$, 
$$
\dim \HFhat(S^3_{p/q}(K), [i]) = 1 + \mathcal{S}^{p/q}_{[i]}.
$$
\end{theorem}

\begin{proof}[Proof of Proposition~\ref{prop:sameKinequality}]
To prove the proposition, it suffices to mirror the knot $K$ so as to be in the setting of Theorem~\ref{thm:jabuka}, provided that we additionally prove the inequality
\begin{equation}\label{eq:sameK-minus}
\dim \HFhat(S^3_{- \fpq}(K),\spinc;\zz/r\zz) < \dim \HFhat(S^3_{- \fpqprime}(K),\spinc'; \zz/r\zz).
\end{equation}
From now on, we assume the formulas in Theorem~\ref{thm:jabuka} hold for the knot $K$.  

We begin with the following observation.  If $\lceil q/p \rceil < \lfloor q'/p' \rfloor$, we have that $\phi^{p/q}_{[i]}(s) < \phi^{p'/q'}_{[i']}(s)$ and $\phi^{-p/q}_{[i]}(s) < \phi^{-p'/q'}_{[i']}(s)$ for all $[i], [i'], s$ by \eqref{eq:phi-inequality}.

We first consider the case that $\nu = 0$.  As discussed, this implies that $\dim H_*(\widehat{A}_s) > 1$ for some $s$.  Since $\phi^{p/q}_{[i]}(s) < \phi^{p'/q'}_{[i']}(s)$ for all $[i], [i'], s$, we see that $\mathcal{S}^{p/q}_{[i]} < \mathcal{S}^{p'/q'}_{[i']}$ for all $[i], [i']$.     Theorem~\ref{thm:jabuka} now establishes \eqref{eq:sameK}.  The same argument applies to show that $\mathcal{S}^{-p/q}_{[i]} < \mathcal{S}^{-p'/q'}_{[i']}$, and hence we obtain \eqref{eq:sameK-minus}. 

Next, consider the case that $\nu > 0$.  We first analyze the positive surgeries ($p/q, p'/q'$).  Since $\phi^{p/q}_{[i]}(s) < \phi^{p'/q'}_{[i']}(s)$ for all $[i], [i'], s$, we have that $\mathcal{S}^{p/q}_{[i]} \leq \mathcal{S}^{p'/q'}_{[i']}$ for all $[i], [i']$.  Observe that we cannot have $(2\nu - 1) q' \leq p'$ since $\lfloor q'/p' \rfloor > \lceil q/p \rceil$ by assumption and hence $\lfloor q'/p' \rfloor \geq 2$.  Therefore, by Theorem~\ref{thm:jabuka}, in order to prove \eqref{eq:sameK}, it suffices to prove the inequalities 
$$
1 \leq -1 + 2 \sum_{|s| < \nu} \phi^{p/q}_{[i]}(s) < -1 + 2\sum_{|s| < \nu} \phi^{p'/q'}_{[i']}(s),
$$  
for any $[i], [i']$.  These follow from \eqref{eq:phi-inequality}, since $p/q \leq 1$ and $\nu > 0$.  

Now, we consider the case of negative surgeries when $\nu > 0$.  As before, we have that $\mathcal{S}^{-p/q}_{[i]} \leq \mathcal{S}^{-p'/q'}_{[i']}$ for all $[i], [i']$.  By Theorem~\ref{thm:jabuka} it suffices to establish the inequality 
$$
1 + 2 \sum_{|s| < \nu} \phi^{-p/q}_{[i]}(s) < 1 + 2\sum_{|s| < \nu} \phi^{-p'/q'}_{[i']}(s),
$$   
for any $[i], [i']$, which again follows from \eqref{eq:phi-inequality}, since $p/q \leq 1$ and $\nu > 0$. 
\end{proof}

\begin{remark}
Note that if $K$ is hyperbolic, a variant of Theorem~\ref{thm:surgerysameK} can be obtained for generic $p, q, p', q'$ via classical methods as follows.  The following argument was shown to us by John Luecke.  First, fix $\fpq \in \mathbb{Q}$ (not necessarily between 0 and 1).  For generic $\fpqprime$, we will have that $S^3_{\fpqprime}(K)$ is hyperbolic by Thurston's hyperbolic Dehn surgery theorem.  Thus, if $S^3_{\fpq}(K)$ is not hyperbolic, then it cannot cover $S^3_{\fpqprime}$ when the latter is hyperbolic.  If instead, $S^3_{\fpq}(K)$ is hyperbolic, then we have $vol(S^3_{\fpq}(K)) < vol(K)$.  Further, for fixed $\epsilon > 0$, for $p'$ and $q'$ large enough, we have $vol(S^3_{\fpqprime}(K)) \geq vol(K) - \epsilon$; see \cite[Theorem 1A]{NeumannZagier} for explicit bounds in terms of $p',q'$.  In particular, we have $vol(S^3_{\fpq}(K)) < vol(S^3_{\fpqprime}(K))$.  Recall that for hyperbolic manifolds, if $\widetilde{Y}$ covers $Y$, then $vol(\widetilde{Y}) \geq vol(Y)$.  Thus, we have that $S^3_{\fpq}(K)$ cannot cover $S^3_{\fpqprime}(K)$.  

Alternatively, we could fix $\fpqprime$, and allow $\fpq$ to vary.  Fix $\delta > 0$.  For generic $\fpq$, we have that the length of the shortest geodesic in $S^3_{\fpq}(K)$ is at most $\delta$ \cite[Proposition 4.3]{NeumannZagier}.  We thus choose $\fpq$ such that the length of the shortest geodesic in $S^3_{\fpq}(K)$ is less than the length of the shortest geodesic in $S^3_{\fpqprime}(K)$.  We see that in this case $S^3_{\fpq}(K)$ cannot cover $S^3_{\fpqprime}(K)$, since if $\gamma$ was the shortest geodesic, its projection to $S^3_{\fpqprime}(K)$ determines a geodesic in $S^3_{\fpqprime}(K)$ with the same length.  Thus, we would have that the length of the shortest geodesic in $S^3_{\fpqprime}(K)$ is at most that of $S^3_{\fpq}(K)$, which is a contradiction.              

Note that in either setting, we did not need any assumptions on the type of covering.  
\end{remark}

Using a similar argument to the one for Theorem~\ref{thm:surgerysameK}, we can also prove Theorem~\ref{thm:surgerylspaceK}.

\begin{proof}[Proof of Theorem~\ref{thm:surgerylspaceK}]
If $p/q \geq 2g(K) - 1$, then $S^3_{p/q}(K)$ is a $\zz/r\zz$-$L$-space by \cite{Hom}.  Since $K$ is non-trivial, we have
$$
1 \leq 2g(K) - 1 = (2g(K) - 1) \lceil q/p \rceil < (2g(K') - 1) \lfloor q'/p' \rfloor < (2g(K') - 1) q'/p', 
$$
and thus $p'/q' < 2g(K') - 1$.  Therefore, $S^3_{p'/q'}(K')$ is not a $\zz/r\zz$-$L$-space by \cite{KMOS}.  The result now follows from Corollary~\ref{cor:ZpL}.  Therefore, we now assume that $0 <p/q < 2g(K) - 1$.  

It is well-known that for a non-trivial $\zz/r\zz$-$L$-space knot, $\dim H_*(\widehat{A}_s) = 1$ for all $s$.  In this case, $\nu = g(K)$ \cite[Proposition 9.7]{RatSurg}.  Therefore, since $K$ and $K'$ are $L$-space knots, we have $\mathcal{S}^{p/q}_{[i]} = \mathcal{S}^{p'/q'}_{[i']} = 0$ for both $K$ and $K'$.  Therefore, by Theorem~\ref{thm:jabuka}, to establish the result it suffices to show that 
$$\sum_{|s| < g(K) } \phi^{p/q}_{[i]}(s) < \sum_{|s| < g(K')} \phi^{p'/q'}_{[i']}(s).$$  By our assumptions on $p/q, p'/q', g(K), g(K')$, and by \eqref{eq:phi-inequality}, we have 
$$ \sum_{|s| <g(K)} \phi^{p/q}_{[i]}(s) \leq (2g(K) - 1)\lceil q/p \rceil  < (2g(K') - 1) \lfloor q'/p' \rfloor  \leq \sum_{|s| < g(K')} \phi^{p'/q'}_{[i']}(s).$$
\end{proof}

We end with another application of the same techniques.

\begin{corollary} 
Let $K$ be any alternating or Montesinos knot other than the pretzel knots $\pm P(-2,3,2s+1)$ for any positive $s \geq 3$ or a torus knot.  Then, for any $p'/q' \geq 9$, $S^3_{p'/q'}(P(-2,3,7))$ is not an $r^n$-sheeted regular cover of $S^3_{p/q}(K)$ for any $p/q \in \mathbb{Q}$ and prime $r$. 
\end{corollary}
\begin{proof}
By \cite{BakerMoore, LidmanMoore}, the conditions on $K$ guarantee that $S^3_{p/q}(K)$ is not a $\zz/r\zz$-$L$-space-knot.\footnote{The arguments in \cite{BakerMoore} for non-pretzel Montesinos knots work for any coefficients.  The arguments in \cite{LidmanMoore} for pretzel knots are given over $\zz/2\zz$, because for one knot, the knot Floer homology needs to be computed over $\zz/2\zz$.  However, one can deduce from the calculation using universal coefficients and \cite{OSLens} that that knot cannot be a $\zz/r\zz$-$L$-space knot for any $r$.}  On the other hand, $S^3_{p'/q'}(P(-2,3,7))$ is an $L$-space for $p'/q' \geq 9$.  Now apply Theorem~\ref{thm:surgerylspaceK}. 
\end{proof}

\bibliography{biblio}

\begin{thebibliography}{HRRW15}

\bibitem[Ago13]{Agol}
Ian Agol.
\newblock The virtual {H}aken conjecture.
\newblock {\em Doc. Math.}, 18:1045--1087, 2013.
\newblock With an appendix by Agol, Daniel Groves, and Jason Manning.

\bibitem[BC17]{BoyerClay}
Steven Boyer and Adam Clay.
\newblock Foliations, orders, representations, {L}-spaces and graph manifolds.
\newblock {\em Adv. Math.}, 310:159--234, 2017.

\bibitem[BGW13]{BoyerGordonWatson}
Steven Boyer, Cameron~McA. Gordon, and Liam Watson.
\newblock On {L}-spaces and left-orderable fundamental groups.
\newblock {\em Math. Ann.}, 356(4):1213--1245, 2013.

\bibitem[Blo11]{Bloom}
Jonathan~M. Bloom.
\newblock A link surgery spectral sequence in monopole {F}loer homology.
\newblock {\em Adv. Math.}, 226(4):3216--3281, 2011.

\bibitem[BM14]{BakerMoore}
Kenneth~L.. Baker and Allison~H. Moore.
\newblock {Montesinos knots, Hopf plumbings, and L-space surgeries}.
\newblock Preprint, \url{arXiv:1404.7585}, 2014.

\bibitem[Bre72]{BredonBook}
Glen~E. Bredon.
\newblock {\em Introduction to compact transformation groups}.
\newblock Academic Press, New York, 1972.
\newblock {P}ure and Applied Mathematics, Vol. 46.

\bibitem[BRW05]{BoyerRolfsenWiest}
Steven Boyer, Dale Rolfsen, and Bert Wiest.
\newblock Orderable 3-manifold groups.
\newblock {\em Ann. Inst. Fourier (Grenoble)}, 55(1):243--288, 2005.

\bibitem[CGH12a]{CGH3}
Vincent Colin, Paolo Ghiggini, and Ko~Honda.
\newblock The equivalence of {H}eegaard {F}loer homology and embedded contact
  homology {III}: from hat to plus.
\newblock Preprint, \url{arXiv:1208.1526}, 2012.

\bibitem[CGH12b]{CGH1}
Vincent Colin, Paolo Ghiggini, and Ko~Honda.
\newblock The equivalence of {H}eegaard {F}loer homology and embedded contact
  homology via open book decompositions {I}.
\newblock Preprint, \url{arXiv:1208.1074}, 2012.

\bibitem[CGH12c]{CGH2}
Vincent Colin, Paolo Ghiggini, and Ko~Honda.
\newblock The equivalence of {H}eegaard {F}loer homology and embedded contact
  homology via open book decompositions {II}.
\newblock Preprint, \url{arXiv:1208.1077}, 2012.

\bibitem[Con78]{ConleyBook}
Charles Conley.
\newblock {\em Isolated invariant sets and the {M}orse index}, volume~38 of
  {\em CBMS Regional Conference Series in Mathematics}.
\newblock American Mathematical Society, Providence, R.I., 1978.

\bibitem[Flo52]{Floyd}
E.~E. Floyd.
\newblock On periodic maps and the {E}uler characteristics of associated
  spaces.
\newblock {\em Trans. Amer. Math. Soc.}, 72:138--147, 1952.

\bibitem[Flo87]{FloerConley}
Andreas Floer.
\newblock A refinement of the {C}onley index and an application to the
  stability of hyperbolic invariant sets.
\newblock {\em Ergodic Theory Dynam. Systems}, 7(1):93--103, 1987.

\bibitem[Gor91]{GordonICM}
Cameron~McA. Gordon.
\newblock Dehn surgery on knots.
\newblock In {\em Proceedings of the {I}nternational {C}ongress of
  {M}athematicians, {V}ol.\ {I}, {II} ({K}yoto, 1990)}, pages 631--642. Math.
  Soc. Japan, Tokyo, 1991.

\bibitem[Hen12]{Hendricks}
Kristen Hendricks.
\newblock A rank inequality for the knot {F}loer homology of double branched
  covers.
\newblock {\em Algebr. Geom. Topol.}, 12(4):2127--2178, 2012.

\bibitem[Hom11]{Hom}
Jennifer Hom.
\newblock A note on cabling and {$L$}-space surgeries.
\newblock {\em Algebr. Geom. Topol.}, 11(1):219--223, 2011.

\bibitem[HRRW15]{HRRW}
Jonathan Hanselman, Jacob Rasmussen, Sarah~D. Rasmussen, and Liam Watson.
\newblock Taut foliations on graph manifolds.
\newblock Preprint, \url{arXiv:1508.05911}, 2015.

\bibitem[Jab15]{Jabuka}
S.~Jabuka.
\newblock Heegaard {F}loer groups of {D}ehn surgeries.
\newblock {\em J. Lond. Math. Soc. (2)}, 92(3):499--519, 2015.

\bibitem[JM08]{JabukaMark}
Stanislav Jabuka and Thomas~E. Mark.
\newblock On the {H}eegaard {F}loer homology of a surface times a circle.
\newblock {\em Adv. Math.}, 218(3):728--761, 2008.

\bibitem[KLT10a]{KLT1}
Ca\u{g}atay Kutluhan, Yi-Jen Lee, and Clifford~H. Taubes.
\newblock {HF=HM I : H}eegaard {F}loer homology and {S}eiberg--{W}itten {F}loer
  homology.
\newblock Preprint, \url{arXiv:1007.1979}, 2010.

\bibitem[KLT10b]{KLT2}
Ca\u{g}atay Kutluhan, Yi-Jen Lee, and Clifford~H. Taubes.
\newblock {HF=HM II : R}eeb orbits and holomorphic curves for the
  ech/{H}eegaard-{F}loer correspondence.
\newblock Preprint, \url{arXiv:1008.1595}, 2010.

\bibitem[KLT10c]{KLT3}
Ca\u{g}atay Kutluhan, Yi-Jen Lee, and Clifford~H. Taubes.
\newblock {HF=HM III : H}olomorphic curves and the differential for the
  ech/{H}eegaard {F}loer correspondence.
\newblock Preprint, \url{arXiv:1010.3456}, 2010.

\bibitem[KLT11]{KLT4}
Ca\u{g}atay Kutluhan, Yi-Jen Lee, and Clifford~H. Taubes.
\newblock {HF=HM IV : T}he {S}eiberg-{W}itten {F}loer homology and ech
  correspondence.
\newblock Preprint, \url{arXiv:1107.2297}, 2011.

\bibitem[KLT12]{KLT5}
Ca\u{g}atay Kutluhan, Yi-Jen Lee, and Clifford~H. Taubes.
\newblock {HF=HM V} : {S}eiberg-{W}itten {F}loer homology and handle additions.
\newblock Preprint, \url{arXiv:1204.0115}, 2012.

\bibitem[KM07]{KMbook}
Peter~B. Kronheimer and Tomasz~S. Mrowka.
\newblock {\em Monopoles and three-manifolds}, volume~10 of {\em New
  Mathematical Monographs}.
\newblock Cambridge University Press, Cambridge, 2007.

\bibitem[KMOS07]{KMOS}
Peter~B. Kronheimer, Tomasz~S. Mrowka, Peter~S. Ozsv{\'a}th, and Zoltan
  Szab{\'o}.
\newblock Monopoles and lens space surgeries.
\newblock {\em Ann. of Math. (2)}, 165(2):457--546, 2007.

\bibitem[Lee05]{Lee}
Yi-Jen Lee.
\newblock Heegaard splittings and {S}eiberg-{W}itten monopoles.
\newblock In {\em Geometry and topology of manifolds}, volume~47 of {\em Fields
  Inst. Commun.}, pages 173--202. Amer. Math. Soc., Providence, RI, 2005.

\bibitem[LL08]{LeeLipshitz}
Dan~A. Lee and Robert Lipshitz.
\newblock Covering spaces and {$\mathbb Q$}-gradings on {H}eegaard {F}loer
  homology.
\newblock {\em J. Symplectic Geom.}, 6(1):33--59, 2008.

\bibitem[LM16a]{Equivalence}
Tye Lidman and Ciprian Manolescu.
\newblock The equivalence of two {S}eiberg-{W}itten {F}loer homologies.
\newblock Preprint, \url{arXiv:1603.00582}, 2016.

\bibitem[LM16b]{LidmanMoore}
Tye Lidman and Allison~H. Moore.
\newblock Pretzel knots with {$L$}-space surgeries.
\newblock {\em Michigan Math. J.}, 65(1):105--130, 2016.

\bibitem[LS07]{LiscaStipsicz3}
Paolo Lisca and Andr{\'a}s~I. Stipsicz.
\newblock Ozsv\'ath-{S}zab\'o invariants and tight contact 3-manifolds. {III}.
\newblock {\em J. Symplectic Geom.}, 5(4):357--384, 2007.

\bibitem[LT16]{LipshitzTreumann}
Robert Lipshitz and David Treumann.
\newblock Noncommutative {H}odge-to-de {R}ham spectral sequence and the
  {H}eegaard {F}loer homology of double covers.
\newblock {\em J. Eur. Math. Soc. (JEMS)}, 18(2):281--325, 2016.

\bibitem[Man03]{Spectrum}
Ciprian Manolescu.
\newblock Seiberg-{W}itten-{F}loer stable homotopy type of three-manifolds with
  {$b_1=0$}.
\newblock {\em Geom. Topol.}, 7:889--932 (electronic), 2003.

\bibitem[Mos71]{Moser}
Louise Moser.
\newblock Elementary surgery along a torus knot.
\newblock {\em Pacific J. Math.}, 38:737--745, 1971.

\bibitem[NW15]{NiWu}
Yi~Ni and Zhongtao Wu.
\newblock Cosmetic surgeries on knots in {$S^3$}.
\newblock {\em J. Reine Angew. Math.}, 706:1--17, 2015.

\bibitem[NZ85]{NeumannZagier}
Walter~D. Neumann and Don Zagier.
\newblock Volumes of hyperbolic three-manifolds.
\newblock {\em Topology}, 24(3):307--332, 1985.

\bibitem[OS04a]{GenusBounds}
Peter~S. Ozsv{\'a}th and Zolt{\'a}n Szab{\'o}.
\newblock Holomorphic disks and genus bounds.
\newblock {\em Geom. Topol.}, 8:311--334, 2004.

\bibitem[OS04b]{HolDiskTwo}
Peter~S. Ozsv{\'a}th and Zolt{\'a}n Szab{\'o}.
\newblock Holomorphic disks and three-manifold invariants: properties and
  applications.
\newblock {\em Ann. of Math. (2)}, 159(3):1159--1245, 2004.

\bibitem[OS04c]{HolDisk}
Peter~S. Ozsv{\'a}th and Zolt{\'a}n Szab{\'o}.
\newblock Holomorphic disks and topological invariants for
  closedthree-manifolds.
\newblock {\em Ann. of Math. (2)}, 159(3):1027--1158, 2004.

\bibitem[OS05a]{OSLens}
Peter~S. Ozsv{\'a}th and Zolt{\'a}n Szab{\'o}.
\newblock On knot {F}loer homology and lens space surgeries.
\newblock {\em Topology}, 44(6):1281--1300, 2005.

\bibitem[OS05b]{BrDCov}
Peter~S. Ozsv{\'a}th and Zolt{\'a}n Szab{\'o}.
\newblock On the {H}eegaard {F}loer homology of branched double-covers.
\newblock {\em Adv. Math.}, 194(1):1--33, 2005.

\bibitem[OS08]{IntSurg}
Peter~S. Ozsv{\'a}th and Zolt{\'a}n Szab{\'o}.
\newblock Knot {F}loer homology and integer surgeries.
\newblock {\em Algebr. Geom. Topol.}, 8(1):101--153, 2008.

\bibitem[OS11]{RatSurg}
Peter~S. Ozsv{\'a}th and Zolt{\'a}n Szab{\'o}.
\newblock Knot {F}loer homology and rational surgeries.
\newblock {\em Algebr. Geom. Topol.}, 11(1):1--68, 2011.

\bibitem[Pet09]{Peters}
Thomas Peters.
\newblock On {L}-spaces and non left-orderable 3-manifold groups.
\newblock Preprint, \url{arXiv:0903.4495}, 2009.

\bibitem[Pru99]{Pruszko}
Artur~M. Pruszko.
\newblock The {C}onley index for flows preserving generalized symmetries.
\newblock In {\em Conley index theory ({W}arsaw, 1997)}, volume~47 of {\em
  Banach Center Publ.}, pages 193--217. Polish Acad. Sci., Warsaw, 1999.

\bibitem[Smi38]{SmithInequality}
P.~A. Smith.
\newblock Transformations of finite period.
\newblock {\em Ann. of Math. (2)}, 39(1):127--164, 1938.

\bibitem[SS10]{SeidelSmithLoc}
Paul Seidel and Ivan Smith.
\newblock Localization for involutions in {F}loer cohomology.
\newblock {\em Geom. Funct. Anal.}, 20(6):1464--1501, 2010.

\bibitem[Tau10]{Taubes12345}
Clifford~Henry Taubes.
\newblock Embedded contact homology and {S}eiberg-{W}itten {F}loer cohomology
  {I-V}.
\newblock {\em Geom. Topol.}, 14(5):2497--3000, 2010.

\bibitem[Wis11]{Wise}
Daniel~T. Wise.
\newblock The structure of groups with a quasiconvex hierarchy.
\newblock Preprint, available at
  \url{http://www.math.mcgill.ca/wise/papers.html}, 2011.

\end{thebibliography}
\bibliographystyle{alpha}

\end{document}